%% file: main.tex
\documentclass[12pt,reqno,a4paper]{amsart}
\usepackage[doi=true,date=year,isbn=false,style=alphabetic,sorting=nyt,backend=biber,maxnames=5,maxalphanames=5,giveninits=true]{biblatex}
\AtEveryBibitem{\clearlist{language}}
\bibliography{main}

\usepackage[breaklinks,colorlinks,plainpages,hypertexnames=false,plainpages=false]{hyperref}
\hypersetup{urlcolor=blue, citecolor=blue, linkcolor=blue}

\usepackage[utf8]{inputenc}
\usepackage{amsmath}
\usepackage{amsthm}
\usepackage{amssymb}
\usepackage{amsfonts}
\usepackage{enumerate}
\usepackage{mathtools}
\usepackage{hhline}
\usepackage{stmaryrd}
\usepackage{enumitem}
\usepackage{centernot}
\usepackage[textwidth=25mm]{todonotes}
\usepackage{mathrsfs}
\usepackage{comment}
\usepackage{listings}
\usepackage{multicol}
\usepackage[capitalise, noabbrev]{cleveref} %for \cref
\usepackage[noend]{algorithmic} % for the algorithmic environment

\usepackage{mathdots}
\usepackage{cleveref} % for \cref
\usepackage{nicematrix} % for NiceMatrix
\usepackage{mathalpha} % for various math fonts
\DeclareMathAlphabet{\dutchcal}{U}{dutchcal}{m}{n}

\usepackage{array}
\newcolumntype{L}[1]{>{\raggedright\let\newline\\\arraybackslash\hspace{0pt}}m{#1}}
\newcolumntype{C}[1]{>{\centering\let\newline\\\arraybackslash\hspace{0pt}}m{#1}}
\newcolumntype{R}[1]{>{\raggedleft\let\newline\\\arraybackslash\hspace{0pt}}m{#1}}

\usepackage{tikz}
\usetikzlibrary{arrows}
\usetikzlibrary{decorations.pathreplacing}
\usetikzlibrary{matrix}
\usetikzlibrary{calc}
\usetikzlibrary{shapes}
\usetikzlibrary{patterns}
\usetikzlibrary{fit,backgrounds,scopes}
\newtheoremstyle{theoremstyle}
{10pt}      %  Space above
{5pt}       %  Space below
{\itshape}  %  Body font
{}          %  Indent amount (empty = no indent, \parindent = para indent)
{\bfseries} %  Thm head font
{}         %  Punctuation after thm head
{ }      %  Space after thm head: " " = normal interword space;
{}          %  Thm head spec (can be left empty, meaning `normal')

\newtheoremstyle{algorithmstyle}
{10pt}      %  Space above
{5pt}       %  Space below
{}  %  Body font
{}          %  Indent amount (empty = no indent, \parindent = para indent)
{\bfseries} %  Thm head font
{}         %  Punctuation after thm head
{ }      %  Space after thm head: " " = normal interword space;
{}          %  Thm head spec (can be left empty, meaning `normal')

\newtheoremstyle{examplestyle}
{10pt}      %  Space above
{5pt}       %  Space below
{}          %  Body font
{}          %  Indent amount (empty = no indent, \parindent = para indent)
{\bfseries} %  Thm head font
{}         %  Punctuation after thm head
{ }      %  Space after thm head: " " = normal interword space;
{}          %  Thm head spec (can be left empty, meaning `normal')

\makeatletter
\newtheorem*{rep@theorem}{\rep@title}
\newcommand{\newreptheorem}[2]{%
\newenvironment{rep#1}[1]{%
 \def\rep@title{#2 \ref{##1}}%
 \begin{rep@theorem}}%
 {\end{rep@theorem}}}
\makeatother

\makeatletter % subalign = substack + align
\newcommand{\subalign}[1]{%
  \vcenter{%
    \Let@ \restore@math@cr \default@tag
    \baselineskip\fontdimen10 \scriptfont\tw@
    \advance\baselineskip\fontdimen12 \scriptfont\tw@
    \lineskip\thr@@\fontdimen8 \scriptfont\thr@@
    \lineskiplimit\lineskip
    \ialign{\hfil$\m@th\scriptstyle##$&$\m@th\scriptstyle{}##$\hfil\crcr
      #1\crcr
    }%
  }%
}
\makeatother

\newreptheorem{theorem}{Theorem}
\newreptheorem{corollary}{Corollary}
\newreptheorem{proposition}{Proposition}

\theoremstyle{theoremstyle}
\newtheorem{theorem}{Theorem}[section]
\newtheorem{lemma}[theorem]{Lemma}

\theoremstyle{examplestyle}

\newtheorem*{notation*}{Notation}

\newtheorem{convention}[theorem]{Convention}

\newtheorem{question}[theorem]{Question}
\theoremstyle{algorithmstyle}

\newcommand{\RR}{\mathbb{R}}

\newcommand{\ZZ}{\mathbb{Z}}

\newcommand{\suchthat}{\;\ifnum\currentgrouptype=16 \middle\fi|\;}

\DeclareMathOperator{\Conv}{Conv}

\DeclareMathOperator{\MV}{MV}

\DeclareMathOperator{\Span}{Span}

\DeclareMathOperator{\trop}{trop}
\DeclareMathOperator{\val}{val}

\newcommand\restr[2]{{\left.\kern-\nulldelimiterspace #1 \right|_{#2}}}

\makeatletter
\newcommand{\oset}[3][0ex]{%
  \mathrel{\mathop{#3}\limits^{
    \vbox to#1{\kern-2\ex@
    \hbox{$\scriptstyle#2$}\vss}}}}
\makeatother

\makeatletter
\newcommand{\uset}[3][0ex]{%
  \mathrel{\mathop{#3}\limits_{
    \vbox to#1{\kern-7\ex@
    \hbox{$\scriptstyle#2$}\vss}}}}
\makeatother
\makeatletter
\newcommand{\customlabel}[2]{%
   \protected@write \@auxout {}{\string \newlabel {#1}{{#2}{\thepage}{#2}{#1}{}} }%
   \hypertarget{#1}{#2}%
}
\makeatother
\begin{document}

\title[]{On intersections and stable intersections\\ of tropical hypersurfaces}

\author{Yue Ren}
\address{Department of Mathematics, Durham University, United Kingdom.
}
\email{yue.ren2@durham.ac.uk}
\urladdr{https://yueren.de}

\subjclass[2020]{14T10, 14T15}

\date{\today}

\keywords{tropical geometry, tropical hypersurfaces, stable intersections.}

\begin{abstract}
  We prove that every connected component of an intersection of tropical hypersurfaces contains a point of their stable intersection unless their stable intersection is empty. This is done by studying algebraic hypersurfaces that tropicalize to them and the tropicalization of their intersection.
\end{abstract}

\thanks{Yue Ren is supported by UK Research and Innovation under the Future Leaders Fellowship programme (MR/S034463/2).}

\maketitle

\input{introduction.tex}

\input{background.tex}
\pagebreak
\input{theorem.tex}

\input{open.tex}

\renewcommand{\emph}[1]{\textit{\textcolor{red}{#1}}}
\renewcommand{\emph}[1]{\textit{#1}}
\renewcommand*{\bibfont}{\small}
\printbibliography

\end{document}

%% file: introduction.tex
\section{Introduction}

Tropical varieties are commonly described as combinatorial shadows of algebraic varieties. The tropicalization of an algebraic variety shares many common properties with its algebraic counterparts, such as its dimension. This is prominently used in the finiteness proof of central configurations in the four and five body problem of celestial mechanics \cite{HamptonMoeckel06,HamptonJensen11}, as well as finiteness proofs of many other central configurations such as those with fixed subconfigurations~\cite{HamptonJensen15} or equilateral chains~\cite{DengHampton23}. In these proofs, the authors exploit the fact that their central configurations satisfy certain algebraic equations and thus lie on an algebraic variety. They then show that the tropicalization of the algebraic variety is zero-dimensional by intersecting the tropical hypersurfaces of carefully chosen equations and eliminating all resulting positive-dimensional polyhedra.

Computing an intersection of tropical hypersurfaces can however be an incredibly challenging task. Tropical hypersurfaces may have several maximal polyhedra and the computation requires intersecting all combinations thereof. While parallelisation and a clever choice of intersection order can lead to significant improvements in performance \cite{JSV17}, there is no general way to avoid the exponential number of intersections required. This circumstance is especially unsatisfying when one expects the final intersection to be very small, such as in all the cases above.

One alternative would be to compute the intersection ``bottom-up'': if one can identify a point in every connected component of the intersection, the rest can be obtained using a traversal as in the computation of tropical varieties \cite{BJST07,MR20}. When the number of hypersurfaces equals the ambient dimension, a natural candidate for such a set of starting points is their stable intersection. This is in part due to similarity of definitions, see \cref{fig:stableIntersection}, but also because it can be computed quickly using techniques such as tropical homotopy continuation \cite{Jensen16}. However, there is no known proof that every connected component of the intersection contains a point in the stable intersection. This paper aims to close that gap:

\begin{reptheorem}{thm:main}
  Let $\Sigma_1,\dots\Sigma_k$ be tropical hypersurfaces in $\RR^n$ with a non-empty stable intersection $\bigwedge_{i=1}^k\Sigma_i\neq\emptyset$. Then every connected component of their intersection $\bigcap_{i=1}^k|\Sigma_i|$ contains a point in the support of their stable intersection $|\bigwedge_{i=1}^k\Sigma_i|$.
\end{reptheorem}

Despite the obvious combinatorial nature of the statement, we were unable to proof the statement using combinatorics alone. Our proof relies on an algebro-geometric result by Josephine Yu on generic polynomials generating prime ideals \cite{Yu16}, and the properties of tropicalizations of irreducible varieties.

\subsection*{Acknowledgements} The work was partially done during Collaborate@ICERM ``\emph{Numerical Algebraic Geometry and Tropical Geometry}'' with Tianran Chen (Auburn), Paul Helminck (Durham), Anders Jensen (Aarhus), Anton Leykin (Georgia Tech) and Josephine Yu (Georgia Tech). The author would like to thank them for helpful discussions, and the institute for its hospitality.

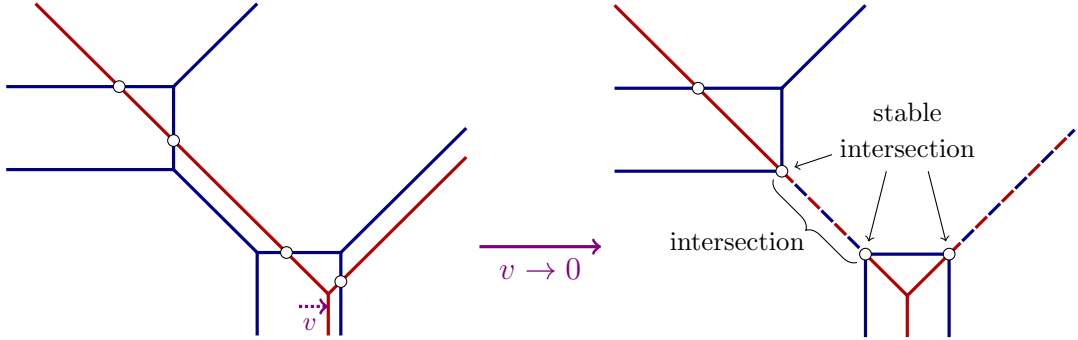
\begin{figure}
  \centering

  \begin{tikzpicture}
    \node at (-4,0)
    {
      \begin{tikzpicture}[scale=1.1]
        \draw[blue!50!black,very thick]
        (-1,2) -- ++(1,1)
        (-1,2) -- ++(-2,0)
        (-1,2) -- (-1,1)
        (-1,1) -- ++(-2,0)
        (-1,1) -- (0,0)
        (0,0) -- ++(0,-1)
        (0,0) -- (1,0)
        (1,0) -- ++(0,-1)
        (1,0) -- ++(1.5,1.5);

        \draw[red!70!black,very thick]
        (0.85,-0.5) -- ++(-3.5,3.5)
        (0.85,-0.5) -- ++(0,-0.5)
        (0.85,-0.5) -- ++(1.65,1.65);

        \draw[fill=white]
        (-1.65,2) circle (2pt)
        (-1,1.35) circle (2pt)
        (0.35,0) circle (2pt)
        (1,-0.35) circle (2pt);

        \draw[->,densely dotted,very thick,violet]
        (0.5,-0.65) -- node[below,font=\footnotesize,xshift=-0.5mm] {$v$} (0.85,-0.65);
      \end{tikzpicture}
    };
    \draw[->,very thick,violet] (-0.8,-1) -- node[below] {$v\rightarrow 0$} (0.8,-1);
    \node at (4,0)
    {
      \begin{tikzpicture}[scale=1.1]
        \draw[blue!50!black,very thick]
        (-1,2) -- ++(1,1)
        (-1,2) -- ++(-2,0)
        (-1,2) -- (-1,1)
        (-1,1) -- ++(-2,0)
        (0,0) -- ++(0,-1)
        (0,0) -- (1,0)
        (1,0) -- ++(0,-1);

        \draw[red!70!black,very thick]
        (-1,1) -- ++(-2,2)
        (0,0) -- (0.5,-0.5)
        (0.5,-0.5) -- ++(0,-0.5)
        (0.5,-0.5) -- (1,0);

        \draw[blue!50!black,very thick,dash pattern= on 6pt off 8pt,dash phase=7pt]
        (-1,1) -- (0,0)
        (1,0) -- ++(1.5,1.5);
        \draw[red!70!black,very thick,dash pattern= on 6pt off 8pt]
        (-1,1) -- (0,0)
        (1,0) -- ++(1.5,1.5);

        \draw[fill=white]
        (-2,2) circle (2pt)
        (-1,1) circle (2pt)
        (0,0) circle (2pt)
        (1,0) circle (2pt);
        \draw[decorate,decoration={brace,amplitude=5pt,mirror}] ($(-1,1)+(-0.1,-0.1)$) -- node[font=\footnotesize,anchor=north east] {intersection} ($(0,0)+(-0.1,-0.1)$);
        \node[font=\footnotesize,text width=20mm,align=center] at (0.5,1.5) {stable\\ intersection};
        \draw[<-,shorten <=5pt] (0,0) -- ++(0.25,1);
        \draw[<-,shorten <=5pt] (1,0) -- ++(-0.25,1);
        \draw[<-,shorten <=5pt] (-1,1) -- ++(0.6,0.2);
      \end{tikzpicture}
    };
  \end{tikzpicture}\vspace{-3mm}
  \caption{The intersection and stable intersection of two tropical plane curves. Notice how every connected component of the intersection contains a point in the stable intersection.}
  \label{fig:stableIntersection}
\end{figure}

%%% Local Variables:
%%% mode: latex
%%% TeX-master: "main"
%%% End:

%% file: background.tex
\section{Background}

In this article, we closely follow the notation of \cite{MaclaganSturmfels15}. In particular:

\begin{convention}
  For the remainder of the paper, we will fix an algebraically closed field $K$ of characteristic $0$ with non-trivial valuation $\val\colon K^\ast\rightarrow\RR$ and an element $t\in K^\ast$ with $\val(t)=1$. Let $K[x^{\pm1}]\coloneqq K[x_1^{\pm1},\dots,x_n^{\pm1}]$ be a multivariate (Laurent) polynomial ring thereover.
\end{convention}

For the sake of brevity, we will abbreviate ``pure, weighted, balanced polyhedral complex'' by ``balanced polyhedral complex'', and we will consider tropical hypersurfaces as balanced polyhedral complexes instead of supports thereof. Additionally, we will denote the stable intersection by ``$\wedge$'' for better inline formatting.

We will further assume some familiarity with the basic concepts in Sections 2 and 3 of \cite{MaclaganSturmfels15}, such as the duality of tropical hypersurfaces and regular subdivisions of Newton polytopes, tropicalizations of algebraic varieties, stable intersections of balanced polyhedral complexes, and mixed volumes of polytopes. In particular, we will be relying heavily on the following results:\vspace{2mm}

\begin{theorem}[{\cite[Theorem 3.3.5]{MaclaganSturmfels15}} Structure Theorem for Tropical Varieties]\label{thm:StructureTheorem}
  Let $X$ be an irreducible variety in $(K^\ast)^n$ of dimension $d$. Then $\trop(X)$ is the support of a balanced polyhedral complex in $\RR^n$ of dimension $d$ that is connected in codimension $1$.
\end{theorem}\vspace{2mm}

\begin{theorem}[{\cite[Theorem 3.6.1]{MaclaganSturmfels15}}]\label{thm:stableTropicalization}
  Let $\Sigma_1,\Sigma_2$ be two balanced polyhedral complex in $\RR^n$ whose support is the tropicalization of two varieties $X_1,X_2\subseteq (K^\ast)^n$. Then there is a Zariski dense subset $U\subseteq (K^\ast)^n$ consisting of elements with component-wise valuation $0$ such that
  \[ |\Sigma_1\wedge\Sigma_2| = \trop(X_1\cap t\cdot X_2) \qquad \text{for all }t\in U. \]
\end{theorem}\vspace{2mm}

\begin{theorem}[{\cite[Theorem 3.6.10]{MaclaganSturmfels15}}]\label{thm:stableIntersection}
  Let $\Sigma_1,\Sigma_2$ be two balanced polyhedral complex in $\RR^n$ of codimension $d,e$, respectively. If the stable intersection $\Sigma_1\wedge\Sigma_2$ is non-empty, then it is a balanced polyhedral complex of codimension $d+e$.
\end{theorem}\vspace{2mm}

Additionally, we will require the following theorem by Yu. Note that {\cite[Theorem 3]{Yu16}} is more general and also covers fields of finite characteristic. We will only need and state its specialisation to fields of characteristic $0$. In the theorem, ``general polynomials'' means polynomials whose coefficients lie in a fixed Zariski open, dense set of the coefficient space.\vspace{2mm}

\begin{theorem}[{\cite[Theorem 3]{Yu16}}]\label{thm:Yu}
  Let $A_1,\dots,A_k\subseteq\ZZ^n$, $k\leq n$, and $0\in A_i$ for all $i=1,\dots,k$. General polynomials in $K[x]$ with monomial supports $A_1,\dots,A_i$ generate a proper ideal whose radical is prime if and only if for every $\emptyset\neq J\subseteq [k]$ one of the following holds:
  \begin{enumerate}
  \item\label{enumitem:Yu1} $\dim \Span\bigcup_{j\in J}A_j>|J|$, or
  \item\label{enumitem:Yu2} $\dim \Span\bigcup_{j\in J}A_j=|J|$ and $\MV(\Conv(A_j)\mid j\in J)=1$.
  \end{enumerate}
\end{theorem}

%%% Local Variables:
%%% mode: latex
%%% TeX-master: "main"
%%% End:

%% file: theorem.tex
% \pagebreak
\section{Main Theorem}

In this section, we prove \cref{thm:main}, beginning with two lemmas. In the first lemma, we consider affine subspaces as balanced polyhedral complexes consisting of a single element. We show that if an affine subspace has an empty stable intersection with a balanced polyhedral complex, then so do its translates, see \cref{fig:subspaceIntersection}.

\begin{figure}
  \centering
  \begin{tikzpicture}
    \node at (0,0)
    {
      \begin{tikzpicture}[x={(90:5mm)},y={(0:5mm)},z={(240:5mm)}]
        \draw[very thick]
        (0,0,0)++(0,-1.5,1.5) -- ++(0,-1.5,1.5)
        (0,0,0)++(0,-1.5,-1.5) -- ++(0,-1.5,-1.5)
        (0,0,0) -- (0,2.5,0)
        (0,5,0) -- ++(0,1.5,1.5)
        (0,5,0) -- ++(0,1.5,-1.5);
        \fill[blue!20,opacity=0.6] (-1.5,-1.5,-3) -- (-1.5,-1.5,3) -- (1.5,-1.5,3) -- (1.5,-1.5,-3) -- cycle;
        \fill[blue!20,opacity=0.6] (-1.5,2.5,-3) -- (-1.5,2.5,3) -- (1.5,2.5,3) -- (1.5,2.5,-3) -- cycle;
        \fill[blue!20,opacity=0.6] (-1.5,6.5,-3) -- (-1.5,6.5,3) -- (1.5,6.5,3) -- (1.5,6.5,-3) -- cycle;
        \draw[very thick] (0,2.5,0) -- (0,5,0)
        (0,0,0) -- ++(0,-1.5,1.5)
        (0,0,0) -- ++(0,-1.5,-1.5)
        (0,5,0)++(0,1.5,1.5) -- ++(0,1.5,1.5)
        (0,5,0)++(0,1.5,-1.5) -- ++(0,1.5,-1.5);
        \fill[blue!70!black] (0,-1.5,1.5) circle (2pt)
        (0,-1.5,-1.5) circle (2pt)
        (0,2.5,0) circle (2pt)
        (0,6.5,1.5) circle (2pt)
        (0,6.5,-1.5) circle (2pt);
        \fill
        (0,0,0) circle (2pt)
        (0,5,0) circle (2pt);
        \node[anchor=north west,blue!70!black] at (-1.5,-1.5,3) {$L'$};
        \node[anchor=north west,blue!70!black] at (-1.5,2.5,3) {$L$};
        \node[anchor=north west,blue!70!black] at (-1.5,6.5,3) {$L''$};
        \node[anchor=south] at (0,0.85,0) {$\Sigma$};
      \end{tikzpicture}
    };
    \node at (8,0)
    {
      \begin{tikzpicture}[x={(90:5mm)},y={(0:5mm)},z={(240:5mm)}]
        \fill[blue!20,opacity=0.6] (-1,-3,-3) -- (-1,-3,3) -- (-1,8,3) -- (-1,8,-3) -- cycle;
        \draw[very thick]
        (0,0,0) -- ++(0,-3,3)
        (0,0,0) -- ++(0,-3,-3)
        (0,0,0) -- (0,2.5,0)
        (0,5,0) -- ++(0,3,3)
        (0,5,0) -- ++(0,3,-3)
        (0,2.5,0) -- (0,5,0);
        \fill
        (0,0,0) circle (2pt)
        (0,5,0) circle (2pt);
        \fill[blue!20,opacity=0.6] (1,-3,-3) -- (1,-3,3) -- (1,8,3) -- (1,8,-3) -- cycle;
        \node[anchor=west,blue!70!black] at (-1,8,-3) {$L$};
        \node[anchor=west,blue!70!black] at (1,8,-3) {$L'$};
      \end{tikzpicture}
    };
  \end{tikzpicture}\vspace{-3mm}
  \caption{The stable intersections of a fixed balanced polyhedral complex $\Sigma$ with translates of an affine subspace are either all non-empty (left) or empty (right), but never both.}
  \label{fig:subspaceIntersection}
\end{figure}
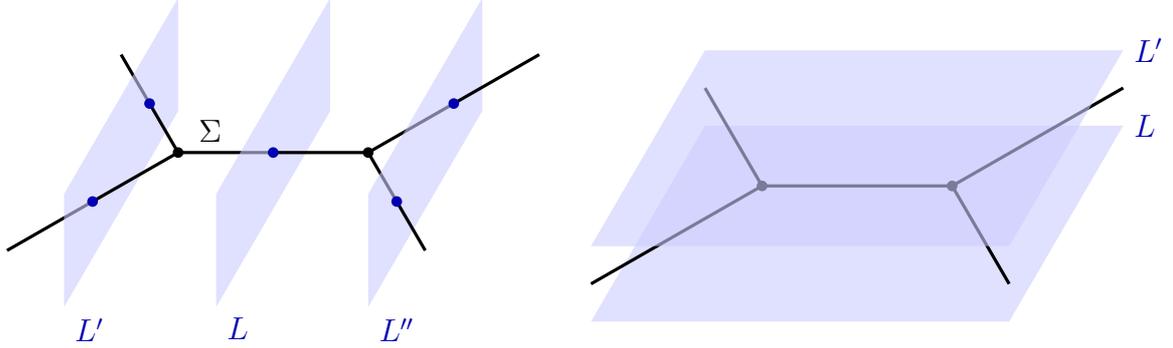

\begin{lemma}\label{lem:subspaceIntersection}
  Let $\Sigma$ be a balanced polyhedral complex in $\RR^n$. Let $L, L'$ be two parallel affine subspaces in $\RR^n$, i.e., $L=\mathrm{Ker}(A)+v$ and $L'=\mathrm{Ker}(A)+v'$ for some matrix $A\in\RR^{k\times n}$ and vectors $v,v'\in\RR^n$. Then
  \[ L\wedge\Sigma=\emptyset\quad\Longleftrightarrow\quad L'\wedge\Sigma=\emptyset. \]
\end{lemma}
\begin{proof}
  Without loss of generality, we may assume that $L$ is a linear space, i.e., $0\in L$. Note that both $L$ and $L'$ can be written as the supports of stable intersections of hyperplanes, i.e., $L=|H_1\wedge\dots\wedge H_k|$ and $L'=|(H_1+v')\wedge H_2\wedge\dots\wedge H_k|$ for suitable hyperplanes $H_1,\dots,H_k$ and $v'\in\RR^n$. Since $H_1+v$ and $H_1+v'$ remain parallel affine subspaces, and $H_2\wedge\dots\wedge H_k\wedge\Sigma$ remains a balanced polyhedral complex, we may assume without loss of generality that both $L$ and $L'$ are affine hyperplanes.

  Consider the projection of $|\Sigma|$ to the one-dimensional orthogonal complement~$L^\perp$. The projection of $|\Sigma|$ remains the support of a balanced polyhedral complex, which means it is either a finite number of points or the entire complement.
  If the projection is a finite number of points, then $\Sigma$ is contained in a union of affine hyperplanes parallel to $L$ and $L'$, and we have $L\wedge\Sigma=\emptyset=L'\wedge\Sigma$. If the projection is the entire complement, then we have  $L\wedge\Sigma\neq\emptyset\neq L'\wedge\Sigma$.
\end{proof}

In the next lemma, we consider connected components of balanced polyhedral complexes as balanced polyhedral complexes. We use \cref{lem:subspaceIntersection} to prove a similar statement but on the connected components of a stable intersection of tropical hypersurfaces instead of the translates of an affine subspace.

\begin{lemma}\label{lem:stableConspiracy}
  Let $\Sigma_1,\dots,\Sigma_k,\Sigma_{k+1}$ be tropical hypersurfaces in $\RR^n$ with $\bigwedge_{i=1}^k\Sigma_i\neq\emptyset$. Then for any two connected components $\Gamma,\Gamma'\subseteq \bigwedge_{i=1}^k\Sigma_i$ we have
  \[ \Gamma\wedge\Sigma_{k+1}=\emptyset \quad\Longleftrightarrow\quad \Gamma'\wedge\Sigma_{k+1}=\emptyset. \]
\end{lemma}
\begin{proof}
  Let $A_1,\dots,A_k\subseteq\ZZ^n$ be monomial supports of polynomials with tropical hypersurfaces $\Sigma_1,\dots,\Sigma_k$ and $0\in A_i$ for all $i$. We distinguish between two cases:
  \begin{description}[leftmargin=*]
  \item[Case 1] If $A_1,\dots,A_k$ do not satisfy Conditions \eqref{enumitem:Yu1} or \eqref{enumitem:Yu2} of \cref{thm:Yu}, then there is a subset $J\subseteq [k]$ such that $\dim \Span (\bigcup_{j\in J} A_j)\leq |J|$.
    Note that the lineality space of $\bigwedge_{j\in J}\Sigma_j$ contains the orthogonal complement of $\Span (\bigcup_{j\in J} A_j)$. Therefore, the lineality space of $\bigwedge_{j\in J}\Sigma_j$ has codimension at most $|J|$. And since $\bigwedge_{j\in J}\Sigma_j\supseteq\bigwedge_{i=1}^k\Sigma_i \neq\emptyset$, \cref{thm:stableIntersection} implies that $\bigwedge_{j\in J}\Sigma_j$ has codimension exactly $|J|$. Hence, $\bigwedge_{j\in J}\Sigma_j$ consists of affine subspaces parallel to its lineality space. The claim follows by applying \cref{lem:subspaceIntersection} to $\Sigma\coloneqq\bigwedge_{i\in [k+1]\setminus J}\Sigma_i$ and the affine subspaces of $\bigwedge_{j\in J}\Sigma_j$.
  \item[Case 2] If $A_1,\dots,A_k$ satisfy Conditions \eqref{enumitem:Yu1} or \eqref{enumitem:Yu2} of \cref{thm:Yu}, consider
  \begin{itemize}
  \item $K\cdot x^A\coloneqq \bigoplus_{i=1}^k \bigoplus_{\alpha\in A_i} K\cdot x^{\alpha}$, the vector space of polynomial tuples $(f_1,\dots,f_k)$ where $f_i$ has monomial support $A_i$, and
  \item $K^{|A|}\coloneqq K^{|A_1\cup\dots\cup A_k|}$, their coefficient space.
  \end{itemize}
  In particular, any choice of coefficients $c=(c_{i,\alpha})_{i\in [k], \alpha\in A_i}\in K^{|A|}$ defines a tuple of polynomials $f(c)\coloneqq (f_i(c))_{i\in [k]}\in K\cdot x^A$ with $f_i(c)\coloneqq\sum_{\alpha\in A_i}c_{i,\alpha}x^\alpha$ and vice versa. By \cref{thm:Yu}, there is a Zariski open, dense subset $U\subseteq K^{|A|}$ such that any $c\in U$ yields an $f(c)$ whose entries generate an ideal whose radical is prime. By \cref{thm:StructureTheorem}, the tropicalization $\trop(V(f(c)))\coloneqq \trop(V(\langle f_1(c),\dots,f_k(c)\rangle))$ will be connected for any $c\in U$. We will now show that there is a $c_0\in U$ such that $\trop(V(f(c_0)))=|\bigwedge_{i=1}^k\Sigma_i|$, so that the claim holds trivially as there is only one connected component.

  Pick any $c\in U$. By \cref{thm:stableTropicalization}, there is a Zariski dense $S\subseteq (K^n)^{k-1}$ with $\trop(V(f_1(c))\cap t_2V(f_2(c))\cap\cdots\cap t_kV(f_k(c)))=|\bigwedge_{i=1}^k\Sigma_i|$ for all ${(t_2,\dots,t_k)\in S}$. For any $t=(t_2,\dots,t_k)\in (K^n)^{k-1}$, let $t^{-1}\cdot c\in K^{|A|}$ denote the coefficients of the polynomials $f_1(t^{-1}\cdot c)=\sum_{\alpha\in A_1}c_{1,\alpha} x^\alpha$ and $f_i(t^{-1}\cdot c)=\sum_{\alpha\in A_i}c_{i,\alpha} t_i^{-\alpha}x^\alpha$ for $i>1$, so that $V(f_1(t^{-1}c))=V(f_1(c))$ and $V(f_i(t^{-1}c))=t_i V(f_i(c))$ for $i>1$. Consider the polynomial map
  \[ \varphi\colon\quad (K^n)^{k-1}\rightarrow K^{|A|}, \qquad t\mapsto t^{-1}\cdot c. \]
  The preimage $\varphi^{-1}(U)$ is a Zariski open set, and since $\varphi(1,\dots,1)=c\in U$ it is non-empty. Thus $\varphi^{-1}(U)$ has to intersect the Zariski dense set $S$, and any point in the image of the intersection gives us the desired $c_0\in U$. \qedhere
  \end{description}
\end{proof}

Using \cref{lem:stableConspiracy} we can now prove:

\pgfdeclareverticalshading{howTheHellDoesThisWork}{4cm}{color(0cm)=(white);color(2.95cm)=(white);color(3.25cm)=(blue!20);color(6cm)=(white)}
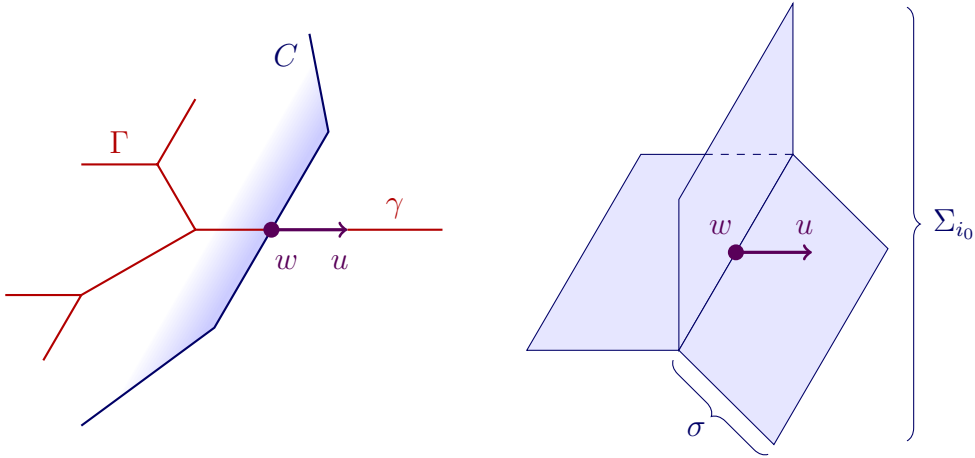
\begin{figure}[t]
  \centering
  \begin{tikzpicture}
    \node (left)
    {
      \begin{tikzpicture}[x={(240:1cm)},y={(0:1cm)},z={(90:1cm)},shading angle from/.style args={line from #1 to #2}{insert path={let \p1=($#2-#1$),\n1={atan2(\y1,\x1)} in},shading angle=\n1}]
        \fill[shading=howTheHellDoesThisWork, shading angle from={line from (-1.5,0,0) to (1.5,0,0)}]  (-2.625,-0.75,0) -- (-1.5,0,0) -- (1.5,0,0) -- (2.625,-0.75,0) -- cycle;
        \draw[very thick, white] (2.625,-0.75,0) -- (-2.625,-0.75,0);
        \draw[thick, blue!40!black] (-3,-1,0) -- (-1.5,0,0) -- (1.5,0,0) -- (3,-1,0);
        \node[anchor=north east, blue!40!black] at (-3,-1,0) {$C$};

        \draw[red!70!black,thick]
        (0,1,0) -- node[above] {$\gamma$} (0,2.25,0)
        (0,0,0) -- (0,-1,0)
        (0,-1,0) -- (1,-2,0)
        (0,-1,0) -- (-1,-2,0)
        (1,-2,0) -- (2,-2,0)
        (1,-2,0) -- (1,-3,0)
        (-1,-2,0) -- node[above] {$\Gamma$} (-1,-3,0)
        (-1,-2,0) -- (-2,-2,0);

        \draw[violet!70!black,->,very thick] (0,0,0) -- (0,1,0);
        \fill[violet!70!black] (0,0,0) circle (3pt);
        \node[anchor=north west,violet!70!black,xshift=-1mm,yshift=-2mm] at (0,0,0) {$w$};
        \node[anchor=north,violet!70!black,xshift=-1mm,yshift=-2mm] at (0,1,0) {$u$};
      \end{tikzpicture}
    };
    \node (right) at (7,0)
    {
      \begin{tikzpicture}[x={(240:1cm)},y={(0:1cm)},z={(90:1cm)}]
        \phantom
        {
          \draw[thick, blue!40!black] (-3,-1,0) -- (-1.5,0,0) -- (1.5,0,0) -- (3,-1,0);
        }
        \draw[blue!40!black,fill=blue!10] (-1.5,-2,0) -- (-1.5,0,0) -- (1.5,0,0) -- (1.5,-2,0) -- cycle;
        \draw[blue!40!black,fill=blue!10] (-1.5,0,2) -- (-1.5,0,0) -- (1.5,0,0) -- (1.5,0,2) -- cycle;
        \draw[blue!40!black,dashed] (-1.5,0,0) -- ++(0,-1.2,0);
        \draw[blue!40!black,fill=blue!10] (-1.5,1.25,-1.25) -- (-1.5,0,0) -- (1.5,0,0) -- (1.5,1.25,-1.25) -- cycle;

        \draw[violet!70!black,->,very thick] (0,0,0) -- (0,1,0);
        \fill[violet!70!black] (0,0,0) circle (3pt);
        \node[anchor=south east,violet!70!black,xshift=1mm,yshift=1mm] at (0,0,0) {$w$};
        \node[anchor=south,violet!70!black,xshift=-1mm,yshift=1mm] at (0,1,0) {$u$};
        \draw[blue!40!black,decorate,decoration={brace,amplitude=5pt,mirror}] (0,2.25,-2.5) -- node[anchor=west,xshift=2mm] {$\Sigma_{i_0}$} (0,2.25,3.25);
        \draw[blue!40!black,decorate,decoration={brace,amplitude=5pt,mirror}] (1.65,0,0) -- node[anchor=north east,xshift=-0.5mm,yshift=-0.5mm] {$\sigma$} (1.65,1.25,-1.25);
      \end{tikzpicture}
    };
    % \fill[white] (left.south east)++(0,0.15) rectangle ++(7,-0.4);
  \end{tikzpicture}\vspace{-2mm}
  \caption{Illustrations for the proof of \cref{thm:main}. $C$ is a connected component of $\bigcap_{i=1}^k|\Sigma_i|$, $\Gamma$ is a connected component of $\bigwedge_{j\in  J}\Sigma_j$ assumed to be not contained in $C$, and $\Sigma_{i_0}$ is a tropical hypersurface that contributes to $\Gamma$ not being contained in $C$.}
  \label{fig:proofIllustration}
\end{figure}

\begin{theorem}\label{thm:main}
  Let $\Sigma_1,\dots\Sigma_k$ be tropical hypersurfaces in $\RR^n$ with a non-empty stable intersection $\bigwedge_{i=1}^k\Sigma_i\neq\emptyset$. Then every connected component of their intersection $\bigcap_{i=1}^k|\Sigma_i|$ contains a point in the support of their stable intersection $|\bigwedge_{i=1}^k\Sigma_i|$.
\end{theorem}
\begin{proof}
  Let $C$ be a connected component of $\bigcap_{i=1}^k|\Sigma_i|$ and assume that $C$ contains no point of the stable intersection $\bigwedge_{i=1}^k\Sigma_i$. As $C$ contains a point of the stable intersection $\bigwedge_{j'\in \{j\}}\Sigma_{j'}=\Sigma_j$ for any singleton $\{j\}\subseteq [n]$, there is a maximal subset $J\subsetneq [k]$ such that $C$ contains a point of the stable intersection $\bigwedge_{j\in J}\Sigma_j$. Let $\Gamma\subseteq|\bigwedge_{j\in J}\Sigma_j|$ be the connected component containing said point. Note that $\Gamma$ is positive dimensional by \cref{thm:stableIntersection}.

  If $\Gamma$ is completely contained in $C$, then $\Gamma\wedge\Sigma_i=\emptyset$ for any $i\notin J$ due to the maximality of $J$. \cref{lem:stableConspiracy} then implies $\bigwedge_{i=1}^k \Sigma_i=\emptyset$, contradicting the assumptions.

  If $\Gamma$ is not completely contained in $C$, then, then there is a point $w\in C\cap\Gamma$ and a direction $u\in\RR^n$ such that $w+\varepsilon\cdot u\in\Gamma$ but $w+\varepsilon\cdot u\notin C$ for $\varepsilon>0$ sufficiently small, see \cref{fig:proofIllustration} left. Since $w+\varepsilon\cdot u\notin C$, we also have $w+\varepsilon\cdot u\notin\bigcap_{i=1}^k|\Sigma_i|$. But because $w+\varepsilon\cdot u\in \bigcap_{j\in J}|\Sigma_j|$, there must be an $i_0\notin J$ such that $w+\varepsilon\cdot u\notin|\Sigma_i|$. Let $\gamma\in\Gamma$ be a polyhedron containing $w+\varepsilon\cdot u$. Let $\sigma\in\Sigma_{i_0}$ be a maximal polyhedron on the boundary of the region of $\RR^n\setminus|\Sigma_{i_0}|$ containing $w+\varepsilon\cdot u$, see \cref{fig:proofIllustration} right. Then  $\dim(\gamma+\sigma)=n$ and hence $w\in\gamma\cap\sigma\in\Gamma\wedge\Sigma_{i_0}$ by the definition of stable intersection. As $w\in C$, this contradicts that $J\subseteq[n]$ is maximal.
\end{proof}

%%% Local Variables:
%%% mode: latex
%%% TeX-master: "main"
%%% End:

%% file: open.tex
\section{Open questions}
Recall that our approach to prove \cref{thm:main} relies on the fact that each $\Sigma_i$ is the tropicalization of an algebraic variety and on dimension arguments to show (non-)emptiness. Consequently, it has two key limitations:

First, it cannot deal with balanced polyhedral complexes which do not arise as tropicalizations of algebraic varieties. Such balanced polyhedral complexes are known to exist. In fact, balanced polyhedral complexes that are tropicalizations of algebraic varieties exhibit particularly nice properties such as higher connectivity \cite{MaclaganYu21}. Hence it is not clear whether \cref{thm:main} generalises to arbitrary balanced polyhedral complexes:

\begin{question}
  Let $\Sigma_1,\dots,\Sigma_k$ be balanced polyhedral complexes in $\RR^n$ with a non-empty stable intersection $\bigwedge_{i=1}^k\Sigma_i\neq\emptyset$. Does every connected component of their intersection $\bigcap_{i=1}^n|\Sigma_i|$ contain a point of their stable intersection $|\bigwedge_{i=1}^k\Sigma_i|$?
\end{question}

Second, it cannot predict where the stable intersection points lie. Recall that the original motivation was to find a way to identify a point on each connected component of an intersection of tropical hypersurfaces. In case the number of tropical hypersurfaces in $\RR^n$ exceeds $n$, their stable intersection is empty by \cref{thm:stableIntersection}. Hence their stable intersection does not provide an easy way to construct such points. Instead, a natural alternative are the stable intersections of any $n$ tropiocal hypersurfaces, see \cref{fig:stableIntersection2}. However, since our techniques do not give any information on where those stable intersection points lie, we do not know where they lie on the intersection of all hypersurfaces:

\begin{question}
  Let $\Sigma_1,\dots,\Sigma_k$ be tropical hypersurfaces in $\RR^n$ with $k>n$ and non-empty stable intersections $\bigwedge_{j\in J}\Sigma_j\neq\emptyset$ for all $J\in\binom{[k]}{n}$. Does every connected component of their intersection $\bigcap_{i=1}^k|\Sigma_i|$ contain a point of a stable intersection $|\bigwedge_{j\in J}\Sigma_j|$ for some $J\in\binom{[k]}{n}$?
\end{question}

\begin{figure}
  \centering
  \begin{tikzpicture}
    \draw[blue!60!black,very thick]
    % (0,0) -- ++(-1,0)
    % (0,0) -- ++(0,-1)
    % (0,0) -- (1,1)
    (1,1) -- (3,1)
    (1,1) -- (1,2)
    (3,1) -- ++(0,-2)
    (3,1) -- (4,2)
    (1,2) -- ++(-3,0)
    (1,2) -- (2,3)
    % (3,2) -- (4,2)
    (4,2) -- (4,3)
    % (2,3) -- (2,4)
    (2,3) -- (4,3)
    % (4,2) -- ++(0,-3)
    % (4,2) -- ++(1,1)
    % (2,4) -- ++(-3,0)
    % (2,4) -- ++(1,1)
    (4,3) -- ++(2,2);
    \draw[red!70!black,very thick]
    (2,2) -- ++(-1,-1)
    (2,2) -- ++(2,0)
    (2,2) -- ++(0,1);
    \draw[orange,very thick]
    (0,0) -- (6,0)
    (2,5) -- (2,4)
    (2,2) -- (2,-1);
    \draw[red!90!black,very thick,dash pattern= on 6pt off 8pt,dash phase=7pt]
    (0,0) -- ++(0,-1)
    (0,0) -- (1,1)
    (4,2) -- (5,2)
    (2,4) -- ++(-4,0)
    (2,4) -- ++(1,1)
    (5,2) -- ++(0,-3)
    (5,2) -- ++(1,1);
    \draw[blue!60!black,very thick,dash pattern= on 6pt off 8pt]
    (0,0) -- ++(0,-1)
    (0,0) -- (1,1)
    (4,2) -- (5,2)
    (2,4) -- ++(-4,0)
    (2,4) -- ++(1,1)
    (5,2) -- ++(0,-3)
    (5,2) -- ++(1,1);
    \draw[blue!70!black,line width=2.75pt,dash pattern= on 4pt off 11pt,dash phase=0pt]
    (0,0) -- ++(-2,0)
    (2,3) -- (2,4);
    \draw[red!70!black,line width=2.75pt,dash pattern= on 4pt off 11pt,dash phase=5pt]
    (0,0) -- ++(-2,0)
    (2,3) -- (2,4);
    \draw[orange,line width=2.75pt,dash pattern= on 4pt off 11pt,dash phase=10pt]
    (0,0) -- ++(-2,0)
    (2,3) -- (2,4);
    \draw[fill=black]
    (1,1) circle (2pt)
    (2,1) circle (2pt)
    (2,2) circle (2pt)
    (3,0) circle (2pt)
    (4,2) circle (2pt)
    (5,2) circle (2pt);
    \draw[fill=white]
    (0,0) circle (2.5pt)
    (2,3) circle (2.5pt)
    (2,4) circle (2.5pt)
    (5,0) circle (2.5pt);
  \end{tikzpicture}\vspace{-2mm}
  \caption{Three tropical plane curves and the stable intersection points of any two curves thereof. Stable intersection points which lie in the intersection of all three curves are highlighted in white.}
  \label{fig:stableIntersection2}
\end{figure}
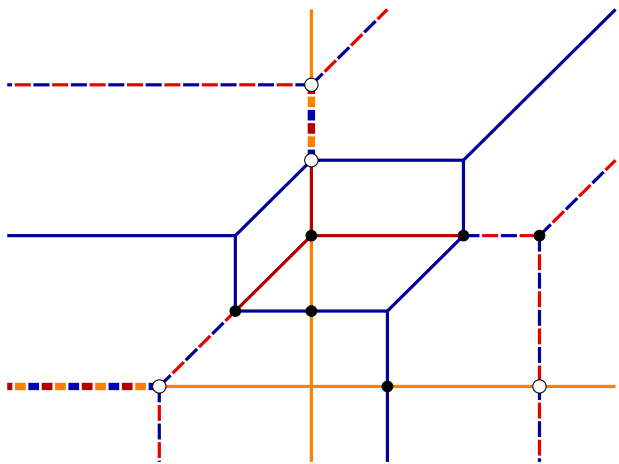

% \subsection*{Complexity of computing intersections} Note that \cref{thm:main} enables the computation of a tropical hypersurface intersection via a traversal. This means, that cells of the hypersurfaces $\Sigma_i$ need not be considered in the computation unless they are adjacent to the intersection. A natural question therefore is whether one can, under sufficiently nice circumstances, bound the complexity of computing the intersection in terms of its (combinatorial) output size.

%%% Local Variables:
%%% mode: latex
%%% TeX-master: "main"
%%% End: